\documentclass[11pt]{article}
\usepackage{amssymb}
\usepackage{amsmath}
\usepackage{amsfonts}
\usepackage{amsthm}
\usepackage{enumerate}
\input amssym.def

\title{\bf Well-posedness for the generalized Benjamin-Ono equations with arbitrary large
initial data in the critical space
 }
\author{St\'ephane Vento, \\Universit\'e Paris-Est, \\Laboratoire d'Analyse
et de Math\'ematiques Appliqu\'ees,\\ 5 bd. Descartes, Cit\'e
Descartes, Champs-Sur-Marne,\\ 77454 Marne-La-Vall\'ee Cedex 2,
France}

\date{E-mail:\, stephane.vento@univ-paris-est.fr}

\numberwithin{equation}{section}

\newtheorem{theorem}{Theorem}[section]
\newtheorem{lemma}{Lemma}[section]
\newtheorem{proposition}{Proposition}[section]

\newtheorem{definition}{Definition}[section]
\newtheorem{remark}{Remark}[section]

\def\R{\mathbb{R}}
\def\Z{\mathbb{Z}}
\def\C{\mathcal{C}}
\def\S{\mathcal{S}}
\def\H{\mathcal{H}}
\def\N{\mathcal{N}}
\def\F{\mathcal{F}}
\def\eps{\varepsilon}
\def\supp{\mathop{\rm supp}\nolimits}
\def\sgn{\mathop{\rm sgn}\nolimits}

\def\B{\dot{\mathcal{B}}}
\def\X{{X^s}}
\def\dX{{\dot{X}^{s}}}

\begin{document}
\maketitle

\noindent {\bf Abstract.}\, We prove that the generalized Benjamin-Ono equations $\partial_tu+\mathcal{H}\partial_x^2u\pm u^k\partial_xu=0$, $k\geq 4$
are locally well-posed in the scaling invariant  spaces $\dot{H}^{s_k}(\R)$ where $s_k=1/2-1/k$. Our results also hold in the non-homogeneous spaces
$H^{s_k}(\R)$. In the case $k=3$, local well-posedness is obtained in $H^{s}(\R)$, $s>1/3$.
\\


\section{Introduction}

In this paper we pursue our study of the Cauchy problem for the  generalized Benjamin-Ono equations
\begin{equation}\label{gBO}\tag{gBO}\left\{\begin{array}{ll}\partial_tu+\mathcal{H}\partial_x^2u\pm u^k\partial_xu=0,
\quad x,t\in\mathbb{R},\\u(x,t=0)=u_0(x),\quad
x\in\mathbb{R},\end{array}\right.\end{equation}
with $k$ an integer $\geq 3$ and with $\H$ the Hilbert transform defined \textit{via} the Fourier transform by
\begin{equation}\label{hilbert}\H f=\F^{-1}(-i\sgn(\xi)\hat{f}(\xi)),\quad f\in\S'(\R).\end{equation}

The Hilbert transform is a real operator, and consequently we look for real-valued solutions.  In view of (\ref{hilbert}), we see that $\H$ is nothing but $-i$ on positive frequencies
and $+i$ on negative ones. A very  close equation to (\ref{gBO}) is then the derivative nonlinear Schr\"{o}dinger equation
\begin{equation}\label{nls}\partial_tu-i\partial_x^2u\pm u^k\partial_xu=0.\end{equation}
for which all our results remain true. Furthermore, (\ref{gBO}) and (\ref{nls}) enjoy the same linear estimates, see Section \ref{sec-lin}.

\vskip 0.5cm

A remarkable feature of (\ref{gBO}) is the following scaling invariance: if $u(t,x)$ is a solution of the equation on $[-T,+T]$, then for any $\lambda>0$,
$u_\lambda(t,x)=\lambda^{1/k}u(\lambda^2t,\lambda x)$
also solves (\ref{gBO}) on $[-\lambda^{-2}T,+\lambda^{-2}T]$ with initial data
$u_\lambda(0,x)$  and moreover
\[\|u_\lambda(\cdot,0)\|_{\dot{H}^s}=\lambda^{s+\frac{1}{k}-\frac{1}{2}}\|u(\cdot,0)\|_{\dot{H}^s}.\]
Hence the $\dot{H}^s(\R)$ norm is invariant if and only if
$s=s_k=1/2-1/k$ and we may expect well-posedness in $\dot{H}^{s_k}(\R)$.

\vskip 0.5cm

When $k=1$, (\ref{gBO}) is the ordinary Benjamin-Ono equation
derived by Benjamin \cite{1967JFM....29..559B} and later by  Ono
\cite{MR0398275} as a model for one-dimensional waves in deep
water. The Cauchy problem for the Benjamin-Ono equation has been
extensively studied these last years, see
\cite{MR533234,MR1097916,MR847994}. In \cite{MR2052470}, Tao introduced a gauge transformation
 (a kind of Cole-Hopf transformation)
which ameliorate the derivative nonlinearity, and
 get the well-posedness of this equation in $H^s(\R)$,
$s\geq 1$. Recently, combining a
gauge transformation together with a Bourgain's method, Ionescu
and Kenig \cite{MR2291918} shown that one could go down to
$L^2(\R)$, which seems to be the critical space for the Benjamin-Ono equation. Note also that Burq and Planchon \cite{MR2357995}
have obtained well-posedness in $H^s(\R)$, $s>1/4$ by similar methods. It is
worth noticing that all these results have been obtained by
compactness methods. On the other hand,  Molinet, Saut and
 Tzvetkov \cite{MR1885293}  proved that, for all $s\in\R$, the
flow map $u_0\mapsto u$ is not of class $\mathcal{C}^2$ from
$H^s(\R)$ to $H^s(\R)$. Furthermore, building suitable families of
approximate solutions,   Koch and  Tzvetkov   proved in
\cite{MR2172940} that the flow map is actually not even uniformly
continuous on bounded sets of $H^s(\R)$, $s>0$.  This explains why a Picard iteration scheme
fails to solve the Benjamin-Ono equation in Sobolev spaces.

\vskip 0.5cm

In the case of the modified Benjamin-Ono equation ($k=2$),
Kenig and  Takaoka \cite{MR2219229} have recently obtained the
global well-posedness in the energy space $H^{1/2}(\R)$. This have
been proved thanks to a localized gauge transformation combined
with a space-time $L^2$ estimate of the solution. It is important to note that this result is far from that given by the scaling index $s_2=0$.
However,  it is known
to be sharp since the solution map $u_0\mapsto u$ is not
$\mathcal{C}^3$ in $H^s(\R)$ as soon as $s<1/2$ (see \cite{MR2038121}).

\vskip 0.5cm

In the case $k=3$, the local
well-posedness is known  in $H^s(\R)$, $s>1/3$ for small initial
data \cite{MR2038121} but only in $H^s(\mathbb{R})$, $s>3/4$ for
large initial data. In \cite{vento-2007}, we showed that (\ref{gBO}) is $\C^4$-ill-posed in $H^s(\R)$, $s<1/3$,
 in the sense that the flow-map $u_0\mapsto u$
fails to be $\C^4$. We prove here that well-posedness occurs in $H^s(\R)$, $s>1/3$, and without smallness
assumption on the initial data.
\vskip 0.5cm

Concerning the case $k\geq 4$, global well-posedness in $H^s(\R)$, $s>s_k$ was derived for small initial data
by Molinet and Ribaud in \cite{MR2038121}. Later, by means of a gauge transformation, the same authors \cite{MR2101982} removed the size restriction on the data and showed well-posedness
in $H^{1/2}(\R)$, whatever the value of $k$. By a refinement of their method, we reached in \cite{vento-2007} the well-posedness in $H^s(\R)$, $s>s_k$,
but for high nonlinearities only ($k\geq 12$ in fact). On the other hand, in the particular case $k=4$, Burq and Planchon \cite{MR2227135} proved the local well-posedness
in the critical space $\dot{H}^{1/4}(\R)$. Inspired by their works, we extend in this paper the well-posedness to $\dot{H}^{s_k}(\R)$ for any $k\geq 4$, and
our method is flexible enough to get the result in the non-homogeneous space
$H^{s_k}(\R)$. A standard fixed point argument allows us to construct a unique solution in a subspace of $\dot{H}^{s_k}(\R)$ with a continuous flow-map $u_0\mapsto u$.
Recall that Biagioni and Linares \cite{MR1837253} proved using solitary waves, that this map cannot be uniformly continuous in $\dot{H}^{s_k}(\R)$.
In the surcritical case $s<s_k$, we also know that the solution-map (if it exists) fails to be $\C^{k+1}$ in $H^s(\R)$, see \cite{MR2038121}.

\section{Notations and main results}
\subsection{Notations}
For $A$ and $B$ two positive numbers, we write $A\lesssim B$ if it exists $c>0$ such that $A\leq cB$. Similarly define $A\gtrsim B$, $A\sim B$
if $A\geq cB$ and $A\lesssim B\lesssim A$ respectively. When the constant $c$ is large enough, we write $A\ll B$. For any $f\in\S'(\R)$, we use
$\F f$ or $\hat{f}$ to denote its Fourier transform. For $1\leq p\leq\infty$, $L^p$ is the standard Lebesgue space and its space-time versions
$L^p_xL^q_T$ and $L^q_TL^p_x$ ($T>0$) are endowed with the norms
$$\|f\|_{L^p_xL^q_T}=\big\|\|f\|_{L^q_t([-T;T])}\big\|_{L^p_x(\R)}\ \textrm{ and
}\
\|f\|_{L^q_TL^p_x}=\big\|\|f\|_{L^p_x(\R)}\big\|_{L^q_t([-T;T])}.$$
The
pseudo-differential operator $D^\alpha_x$ is defined by its
Fourier symbol $|\xi|^\alpha$. We will denote by $P_+$ and $P_-$ the
projection on respectively the positive and the negative spatial Fourier modes. Thus one has
$$i\mathcal{H}=P_+-P_-.$$
Let $\eta\in\mathcal{C}_0^\infty(\R)$, $\eta\geq 0$, $\supp\
\eta\subset\{1/2\leq |\xi|\leq 2\}$ with
$\sum_{-\infty}^\infty\eta(2^{-j}\xi)=1$ for $\xi\neq 0$. We set
$p(\xi)=\sum_{j\leq -3}\eta(2^{-j}\xi)$ and consider, for all
$j\in\Z$, the operator $Q_j$ defined by
\[Q_j(f)=\mathcal{F}^{-1}(\eta(2^{-j}\xi)\hat{f}(\xi)).\]
We adopt the following summation convention. Any summation of the form $r\lesssim j$, $r\gg j$,...
is a sum over the $r\in\Z$ such that $2^r\lesssim 2^j$..., thus for instance $\sum_{r\lesssim j}=\sum_{r : 2^r\lesssim 2^j}$.
We define then the operators  $Q_{\lesssim j}=\sum_{r\lesssim j}Q_r$, $Q_{\ll
j}=\sum_{r\ll j}Q_r$, etc. For $1\leq p,q, r\leq \infty$ and $s\in\R$, let $\B^{s,r}_p(L^q_T)$ be the homogeneous  Besov space equipped with the norm
$$\|f\|_{\B^{s,r}_p(L^q_T)} = \Big(\sum_{j\in\Z}[2^{js}\|Q_jf\|_{L^p_xL^q_T}]^r\Big)^{1/r}.$$
Finally for $s\in\R$ and $\theta\in[0,1]$, we define the solution space $\dot{\S}^{s,\theta}$ (where lives our solution $u$) and the nonlinear space
$\dot{\N}^{s,\theta}$ (where lives the nonlinear term $u^k\partial_xu$) by
$$\dot{\S}^{s,\theta}=\B^{s+\frac{3\theta-1}{4},2}_{\frac{4}{1-\theta}}(L^{\frac{2}{\theta}}_T),\quad
\dot{\N}^{s,\theta}=\B^{s+\frac{1-3\theta}{4},2}_{\frac{4}{3+\theta}}(L^{\frac{2}{2-\theta}}_T).$$

\subsection{Main results}
We first state our well-posedness results in the case $k\geq 4$.
\begin{theorem}\label{th-hom} Let $k\geq 4$ and $u_0\in \dot{H}^{s_k}(\R)$. There
exists $T=T(u_0)>0$ and a unique solution $u$ of (\ref{gBO}) such
that $u\in \dot{Z}_T$ with
$$\dot{Z}_T=\C([-T,+T],\dot{H}^{s_k}(\R))\cap \dot{X}^{s_k}\cap L^k_xL^\infty_T.$$
Moreover, the flow map $u_0\mapsto u$ is locally Lipschitz from $\dot{H}^{s_k}(\R)$ to $\dot{Z}_T$.
\end{theorem}

\vskip 0.3cm

In the non-homogeneous case, one has the following result.

\begin{theorem}\label{th-nohom} Let $k\geq 4$ and $u_0\in {H}^{s}(\R)$, $s\geq s_k$. There
exists $T=T(u_0)>0$ and a unique solution $u$ of (\ref{gBO}) such
that $u\in Z_T$ with
$$Z_T=\C([-T,+T],{H}^{s}(\R))\cap \X\cap L^k_xL^\infty_T.$$
Moreover, the flow map $u_0\mapsto u$ is locally Lipschitz from $H^s(\R)$ to $Z_T$.
\end{theorem}

\begin{remark} We only obtain the Lipschitz continuity of
the map $u_0\mapsto u$ in Theorems \ref{th-hom} and \ref{th-nohom} in $\dot{H}^{s_k}(\R)$ (resp. $H^s(\R)$).
As noticed in the introduction, the solution map given by Theorem \ref{th-hom} is not uniformly continuous from $\dot{H}^{s_k}(\R)$ to
$\C([-T,T],\dot{H}^{s_k}(\R))$. Moreover, when $s<s_k$, the flow map in Theorem \ref{th-nohom}  is no longer of class $\C^{k+1}$ in $H^s(\R)$.
It is not clear wether the map given by Theorems \ref{th-hom} and \ref{th-nohom} is $\C^{k+1}$ or not.
\end{remark}

\begin{remark} The spaces $\dot{X}^{s_k}$ and $X^s$ will be defined in Section \ref{sec-lin} and are directly related with the linear estimates
for the linear Benjamin-Ono equation.
\end{remark}

\vskip 0.3cm

The main tools to prove Theorems \ref{th-hom} and \ref{th-nohom} are the sharp Kato smoothing effect and the maximal in time
inequality for the free solution $V(t)u_0$ where $V(t)=e^{it\H\partial_x^2}$. Recall that for regular solutions, (\ref{gBO}) is equivalent to its integral
formulation
\begin{equation}\label{eq-int}u(t)=V(t)u_0\mp\int_0^tV(t-t')(u^k(t')\partial_xu(t'))dt'.\end{equation}
It is worth noticing that (\ref{gBO}) provides a perfect balance between the derivative nonlinear term on one hand, and the available linear estimates on the other hand.
Heuristically, one may use (\ref{eq-int}) to write
\begin{align*}
\|D_x^{s_k+1/2}u\|_{L^\infty_xL^2_T}+\|u\|_{L^k_xL^\infty_T} &\lesssim \|u_0\|_{\dot{H}^{s_k}}+\|D_x^{s_k-1/2}\partial_x(u^{k+1})\|_{L^1_xL^2_T}\\
&\lesssim \|u_0\|_{\dot{H}^{s_k}}+\|D_x^{s_k+1/2}u\|_{L^\infty_xL^2_T}\|u\|_{L^k_xL^\infty_T}^k
\end{align*}
and perform a fixed point procedure. Unfortunately, such an argument fails for several reasons:
\begin{itemize}
\item First, it is not clear wether the second inequality holds true or not. Indeed, we used the fractional Leibniz rule (see the Appendix in
\cite{MR1211741}, \cite{MR2101982})  at the end points $L^p$, $p=1,\infty$. However, this inequality becomes true if
one works in the associated Besov spaces $\B^{s_k+1/2,2}_\infty(L^2_T)\cap \B^{0,2}_k(L^\infty_T)$ and provides
sharp well-posedness for small initial data, see \cite{MR2038121}.
\item The term $\|V(t)u_0\|_{L^k_xL^\infty_T}$ will be small only if $\|u_0\|_{\dot{H}^{s_k}}$ is small as well, even for small $T$.
Nevertheless, as noticed in \cite{MR2227135}, if we consider instead the difference $V(t)u_0-u_0$, then its $L^k_xL^\infty_T$-norm is
small provided we restrict ourselves to a small interval $[-T,T]$ (see Lemma \ref{lemsmall}).
\item We also need to get a better share of the derivative in the nonlinear term. By a standard paraproduct decomposition, we see that
the worst contribution in $\partial_xu^{k+1}$ is given by $\pi(u,u)$ where
$$\pi(f,g)=\sum_j\partial_xQ_j((Q_{\ll j}f)^kQ_{\sim j}g).$$
The main idea is then to inject this term (or more precisely $\pi(V(t)u_0,u)$)
 in the linear part of the equation to get the variable-coefficient Schr\"{o}dinger equation
\begin{equation}\label{gBO2}\partial_tu+\H\partial_x^2u+\pi(V(t)u_0,u)=f\end{equation}
where $f$ will be a well-behaved term. Linear estimates for equation (\ref{gBO2}) are obtained by the localized gauge transform
$$w_j=e^{\frac i2\int_{-\infty}^x(Q_{\ll j}u_0)^k}P_+Q_ju,\quad j\in\Z.$$
\end{itemize}

\vskip 0.5cm

Now we turn to the case $k=3$. By similar considerations, we obtain the following result.
\begin{theorem}\label{th-keq3} Let $k= 3$ and $u_0\in {H}^{s}(\R)$, $s>1/3$. There
exists $T=T(u_0)>0$ and a unique solution $u$ of (\ref{gBO}) such
that $u\in Z_T$ with
$$Z_T=\C([-T,+T],{H}^{s}(\R))\cap \X\cap L^3_xL^\infty_T.$$
Moreover, the flow map $u_0\mapsto u$ is locally Lipschitz from $H^s(\R)$ to $Z_T$.
\end{theorem}

This paper is organized as follows. In Section \ref{sec-lin}, we recall some sharp estimates related with the linear operator $V(t)$, and we derive linear estimates for equation
(\ref{gBO2}). Section \ref{sec-kgeq4} is devoted to the case $k\geq 4$. Finally, we prove Theorem \ref{th-keq3} in Section \ref{sec-keq3}.

\section{Linear estimates}\label{sec-lin}
\subsection{Estimates for the linear BO equation}
This section deals with the well-known linear estimates for the Benjamin-Ono equation. Note that all results
stated here hold as well for the Schr\"{o}dinger operator $S(t)=e^{it\partial_x^2}$.

The following lemma summarizes the main estimates related to the group $V(t)$. See for instance \cite{MR1101221,MR1086966} for the proof.
\begin{lemma}\label{lem-estlin} Let $\varphi\in\mathcal{S}(\R)$, then
\begin{eqnarray}\label{est0}\|V(t)\varphi\|_{L^\infty_TL^2_x} &\lesssim
&\|\varphi\|_{L^2},\\
\label{est1}\|D^{1/2}_xV(t)\varphi\|_{L^\infty_x L^2_T} &\lesssim
&
\|\varphi\|_{L^2},\\
\label{est2}\|D^{-1/4}_xV(t)\varphi\|_{L^4_x L^\infty_T} &\lesssim
& \|\varphi\|_{L^2}.\end{eqnarray} Moreover, if $T\leq 1$ and $j\geq 0$,
\begin{align}
\|Q_{\leq 0}V(t)\varphi\|_{L^2_xL^\infty_T} &\lesssim \|Q_{\leq 0}\varphi\|_{L^2}\\
2^{-j/2}\|Q_jV(t)\varphi\|_{L^2_xL^\infty_T} &\lesssim \|Q_j\varphi\|_{L^2}
\end{align}
\end{lemma}

\begin{definition} A triplet $(\alpha,p,q)\in\R\times[2,\infty]^2$
is said to be
 1-admissible if
$(\alpha,p,q)=(1/2,\infty,2)$ or \begin{equation}\label{1ad} 4\leq
p<\infty,\quad 2< q\leq\infty,\quad
\frac{2}{p}+\frac{1}{q}\leq\frac{1}{2},\quad
\alpha=\frac{1}{p}+\frac{2}{q}-\frac{1}{2}.
\end{equation}
\end{definition}

By Sobolev embedding and interpolation between estimates (\ref{est1}) and (\ref{est2}) we obtain the following result.

\begin{proposition}[\cite{MR2101982}]\label{admis} If $(\alpha,p,q)$ is 1-admissible, then for all $\varphi$ in
$\S(\R)$,
\begin{equation}\label{in1ad}\|D^\alpha_xV(t)\varphi\|_{L^p_xL^q_T}\lesssim
\|\varphi\|_{L^2}.\end{equation}
\end{proposition}

Now we define our resolution spaces.
\begin{definition}
Let $k\geq 4$ and $s\in\R$ be fixed. For $0<\eps\ll 1$, we define the spaces $\dX=\dot{\S}^{s,\eps}\cap\dot{\S}^{s,1}$
endowed with the norm
$$\|u\|_\dX=\|u\|_{\dot{\S}^{s,\eps}}+\|u\|_{\dot{\S}^{s,1}}.$$
\end{definition}

At this stage it is important to remark that $\dX$ does not contain any $L^\infty_T$ component. As a consequence, for each
$u\in\dX$ and $\eta>0$ fixed, we can choose $T=T(u)$ such that $\|u\|_\dX<\eta$.

\vskip 0.3cm

In the case $k=3$, we shall require the following result which is not covered by Proposition \ref{admis}.
\begin{lemma}[\cite{MR2101982}]\label{lem-linl3}
Let $0<T\leq 1$ and $s>1/3$. Then it holds that
\begin{equation}\label{est3}\|V(t)\varphi\|_{L^3_xL^\infty_T}\lesssim
\|\varphi\|_{H^s},\quad \forall \varphi\in\S(\R).\end{equation}
\end{lemma}

We next state the $L^p_xL^q_T$ and $L^q_TL^p_x$ estimates for the linear operator $f\mapsto\int_0^tV(t-t')f(t')dt'$.

\begin{lemma}[\cite{MR2101982}]\label{lem-estnohom} Let $\alpha\in\R$,
 and $2< p,q\leq \infty$
such that for all $\varphi\in\mathcal{S}(\R)$,
\[\|D^{\alpha}_xV(t)\varphi\|_{L^{p}_xL^{q}_T}\lesssim
\|\varphi\|_{L^2}.\]
 Then for all
$f\in\mathcal{S}(\R^2)$,
\begin{equation}\label{estnonhom}\Big\|D^{1/2}_x\int_0^tV(t-t')f(t')dt'\Big\|_{L^\infty_TL^2_x}\lesssim
\|f\|_{L^1_xL^{2}_T},\end{equation}
\begin{equation}\label{estnonhom2}\Big\|D^{\alpha+1/2}_x\int_0^tV(t-t')f(t')dt'\Big\|_{L^{p}_xL^{q}_T}
\lesssim
\|f\|_{L^{1}_xL^{2}_T}.\end{equation}
Similarly, if $$\|D_x^\alpha V(t)\varphi\|_{L^p_xL^q_T}\lesssim \|\varphi\|_{H^s}$$ for any $\varphi\in \mathcal{S}(\R)$, then
\begin{equation}\label{estnonhom3}\Big\|D^{\alpha+1/2}_x\int_0^tV(t-t')f(t')dt'\Big\|_{L^{p}_xL^{q}_T}
\lesssim
\|\langle D_x\rangle^sf\|_{L^{1}_xL^{2}_T}.\end{equation}
\end{lemma}

We shall need the following Besov version of Lemma \ref{lem-estnohom}.

\begin{lemma}\label{lem-nohombes} Let $k\geq 4$.
For all $f\in\S(\R)^2$,
$$\Big\|\int_0^tV(t-t')f(t')dt'\Big\|_{L^k_xL^\infty_T}\lesssim \|f\|_{\dot{\N}^{s_k,1}}.$$
\end{lemma}

\begin{proof}  Note that the triplets $(1/2,\infty,2)$ and $(-s_k,k,\infty)$ are both 1-admissible. In particular we deduce
$$\Big\|\int_{-T}^TD_x^{1/2}V(-t')h(t')dt'\Big\|_{L^2}\lesssim \|h\|_{L^1_xL^2_T},\quad \forall h\in \S(\R^2),$$
which is the dual estimate of (\ref{in1ad}) for $(\alpha,p,q)=(1/2,\infty,2)$. Since $L^2=\B^{0,2}_2$, we infer
$$\Big\|\int_{-T}^TD_x^{1/2}V(-t')h(t')dt'\Big\|_{L^2}\lesssim \|h\|_{\B^{0,2}_1(L^2_T)},\quad \forall h\in \S(\R^2).$$
The usual $TT^\ast$ argument provides
$$\Big\|\int_{-T}^TV(t-t')f(t')dt'\Big\|_{L^k_xL^\infty_T}\lesssim \|f\|_{\B^{-1/k,2}_1(L^2_T)}.$$
We can conclude with the Christ-Kiselev lemma for reversed norms (Theorem B in \cite{MR2227135}).
\end{proof}

\subsection{Linear estimates for equation (\ref{gBO2})}

Here and hereafter we take $k\geq 4$, the special case $k=3$ will be discussed in Section \ref{sec-keq3}.

Next lemma will be crucial in the proof of our main results.

\begin{lemma}\label{lemsmall}
Let $k\geq 4$ and $u_0\in \dot{H}^{s_k}$. For any $\eta>0$, there exists $T=T(u_0)$ such that
$$\|V(t)u_0-u_0\|_{L^k_xL^\infty_T}<\eta.$$
\end{lemma}
\begin{proof} Let $N>0$ to be chosen later. One has
\begin{align*}\|V(t)u_0-u_0\|_{L^k_xL^\infty_T} &\lesssim
\sum_{|j|<N}\|Q_j(V(t)u_0-u_0)\|_{L^k_xL^\infty_T}+\left(\sum_{|j|>N}\|Q_ju_0\|_{\dot{H}^{s_k}}^2\right)^{1/2}.
\end{align*}
Note
that $v=V(t)u_0-u_0$ solves the equation
$$\partial_tv+\H\partial_x^2v=-\H\partial^2_xu_0$$ with zero
initial data. Thus
$V(t)u_0-u_0=\int_0^tV(t-t')\H\partial_x^2u_0dt'$ and
\begin{align*}\sum_{|j|<N}\|Q_j(V(t)u_0-u_0)\|_{L^k_xL^\infty_T} &\lesssim
\sum_{|j|<N}2^{2j}\Big\|\int_0^tV(t')Q_j
u_0dt'\Big\|_{L^k_xL^\infty_T}\\
&\lesssim T\sum_{|j|<N}2^{2j}\|V(t)Q_ju_0\|_{L^k_xL^\infty_T}\\
&\lesssim T2^{2N}\|u_0\|_{\dot{H}^{s_k}}.
\end{align*}
It suffices now to
choose sufficiently large $N$ and then $T$ small enough.
\end{proof}

Let us turn back to the nonlinear (\ref{gBO}) equation. The sign of the nonlinearity is irrelevant in the study of the local problem, and
we choose for convenience the plus sign.

Using standard paraproduct rearrangements, we can rewrite the nonlinear term in (\ref{gBO}) as follows:
\begin{align*}
\partial_xQ_j(u^{k+1}) &= \partial_xQ_j(\lim_{r\rightarrow\infty}(Q_{<r}u)^{k+1})\\
&= \partial_xQ_j\Big(\sum_{-\infty}^\infty (Q_{<r+1}u)^{k+1}-(Q_{<r}u)^{k+1}\Big)\\
&= \partial_xQ_j\Big(\sum_{-\infty}^\infty Q_ru (Q_{\lesssim r}u)^k\Big)\\
&= \partial_xQ_j\Big(\sum_{r\sim j}Q_ru(Q_{\ll r}u)^k\Big)+\partial_xQ_j\Big(\sum_{r\gtrsim j}(Q_{\sim r}u)^2(Q_{\lesssim r}u)^{k-1}\Big)\\
&= \partial_xQ_j((Q_{\ll j}u)^kQ_{\sim j}u)-g_j.
\end{align*}

We set $$\pi(f,g)=\sum_j\partial_xQ_j((Q_{\ll j}f)^kQ_{\sim j}g)$$
so that (\ref{gBO}) reads
$$\partial_tu+\H\partial_x^2u+\pi(u,u)=g(t,x)$$
with $$g=\sum_jg_j.$$ Setting
$$f=\pi(u_L,u)-\pi(u,u)+g$$ where $u_L=V(t)u_0$ is
the solution of the free BO equation, we see that (\ref{gBO}) is equivalent to
\begin{equation}\label{gbo2}\partial_tu+\H\partial_x^2u+\pi(u_L,u)=f(t,x).\end{equation}
We intend to solve (\ref{gBO}) by a fixed point procedure on the Duhamel formulation of (\ref{gbo2}):
$$u(t)=U(t)u_0-\int_0^tU(t-t')f(t')dt',$$
where $U(t)\varphi$ is solution to
$$\partial_tu+\H\partial_x^2u+\pi(V(t)u_0,u)=0,\quad
u(0)=\varphi.$$ It is worth noticing that $U(t)$ depends on the data $u_0$.

\vskip 0.3cm

Setting $u_j=Q_j u$ and $f_j=Q_jf$, we get from
(\ref{gbo2}) that
$$\partial_t u_j+\H\partial_x^2u_j+\partial_x((u_{0,\ll j})^k
u_j)=\partial_x[((u_{0,\ll j})^k-(u_{L,\ll j})^k)
u_j]-\partial_x[Q_j,(u_{L,\ll j})^k]u_{\sim j}+f_j$$ and we will
denote by $R_j$ the right-hand side. Now take the positive
frequencies and set $v_j=P_+u_j$:
$$i\partial_t v_j+\partial_x^2 v_j+i\partial_x((u_{0,\ll j})^k
v_j)=iP_+R_j.$$ With $b_{\ll j}=\frac 12(u_{0,\ll j})^k$, we obtain
\begin{equation}\label{eq-vj}i\partial_t v_j+(\partial_x+ib_{\ll j})^2v_j= g_j\end{equation}
with \begin{equation}\label{gj}g_j=-i\partial_xb_{\ll j}.
v_j-b_{\ll j}^2v_j+iP_+R_j.\end{equation}

\begin{lemma}\label{lem-vj} Let $v_j$ be a solution to (\ref{eq-vj}) with initial data $v_{0,j}\in\dot{H}^{s_k}\cap\dot{H}^s$. Then there exists
$C=C(u_0)$ such that
$$\|v_j\|_\dX\leq C\|v_{0,j}\|_{\dot{H}^s}+C\|g_j\|_{\dot{\N}^{s,1}}.$$
\end{lemma}

\begin{proof} We define $w_j$ by $$w_j=e^{i\int^xb_{\ll j}}v_j.$$ Then we easily check that $w_j$ solves
$$i\partial_tw_j+\partial_x^2w_j=e^{i\int^xb_{\ll j}}g_j.$$
From the well-known linear estimates on the Schr\"{o}dinger equation (Lemmas \ref{lem-estlin}-\ref{lem-estnohom}) we infer
$$\|\partial_xw_j\|_{L^\infty_xL^2_T}\lesssim \|e^{i\int^xb_{\ll j}}v_{0,j}\|_{\dot{H}^{1/2}}+\|g_j\|_{L^1_xL^2_T}.$$
Since $\partial_xw_j=e^{i\int^xb_{\ll j}}(\partial_xv_j+b_{\ll j}v_j)$, we have
\begin{align*}\|\partial_xv_j\|_{L^\infty_xL^2_T} &\lesssim \|\partial_xw_j\|_{L^\infty_xL^2_T}+\|b_{\ll j}v_j\|_{L^\infty_xL^2_T}\\
&\lesssim \|\partial_xw_j\|_{L^\infty_xL^2_T}+2^{-j}\|b_{\ll j}\|_{L^\infty}\|\partial_xv_j\|_{L^\infty_xL^2_T}.
\end{align*}
On the other hand, we can make $2^{-j}\|(u_{0,\ll
j})^k\|_{L^\infty}$ as small as desired by choosing the implicit
constant $J=J(u_0)$ in $u_{0,\ll j}$ large enough:
$$2^{-j}\|(u_{0,<j-J})^k\|_{L^\infty}\lesssim
2^{-j}2^{j-J}\|u_0\|_{L^k}^k\lesssim c(u_0)2^{-J}\ll 1.$$
It follows that $$\|\partial_xv_j\|_{L^\infty_xL^2_T}\lesssim \|e^{i\int^xb_{\ll j}}v_{0,j}\|_{\dot{H}^{1/2}}+\|g_j\|_{L^1_xL^2_T}.$$
We now use the fractional Leibniz rule (Theorem A.12 in \cite{MR1211741}) and Bernstein inequality to estimate the first term in the right-hand side,
\begin{align*}\|e^{i\int^xb_{\ll j}}v_{0,j}\|_{\dot{H}^{1/2}}& \lesssim \|e^{i\int^xb_{\ll j}}\|_{L^\infty}\|v_{0,j}\|_{\dot{H}^{1/2}}+
\|D_x^{1/2}e^{i\int^xb_{\ll j}}\|_{L^\infty}\|v_{0,j}\|_{L^2}\\
&\lesssim \|v_{0,j}\|_{\dot{H}^{1/2}}+\|(u_{0,\ll j})^k\|_{L^2}\|v_{0,j}\|_{L^2}\\
&\lesssim (1+\|u_0\|_{L^k}^k)\|v_{0,j}\|_{\dot{H}^{1/2}}
\end{align*}
Since $v_j$, $g_j$ as well as $v_{0,j}$ are frequency localized, we conclude
\begin{equation}\label{est-vj1}\|v_j\|_{\B^{s+1/2,2}_\infty(L^2_T)} \lesssim \|v_{0,j}\|_{\dot{H}^s}+\|g_j\|_{\B^{s-1/2,2}_1(L^2_T)}.\end{equation}
We also need $L^4_xL^\infty_T$-norm estimates. Our equation can be rewritten as
$$i\partial_t v_j+\partial_x^2v_j=g_j+h_j$$ with
$$h_j=b_{\ll j}^2v_j-i\partial_x(b_{\ll j}v_j)-ib_{\ll
j}\partial_xv_j.$$
Thus we get from Lemmas \ref{lem-estlin}-\ref{lem-estnohom} that
$$\|v_j\|_{\B^{s-1/4,2}_4(L^\infty_T)}\lesssim
\|v_{0,j}\|_{\dot{H}^s}+\|g_j\|_{\B^{s-1/2,2}_1(L^2_T)}+\|h_j\|_{\B^{s-1/2,2}_1(L^2_T)}.$$
We bound the $h_j$ contribution with (\ref{est-vj1}): \begin{align*}\|b_{\ll
j}^2v_j\|_{\B^{s-1/2,2}_1(L^2_T)} &\lesssim
2^{j(s-1/2)}\|b_{\ll j}^2\|_{L^1}\|v_j\|_{L^\infty_xL^2_T}\\
&\lesssim (2^{-j/2}\|b_{\ll
j}\|_{L^2})^2\|v_j\|_{\B^{s+1/2,2}_1(L^2_T)}\\
&\lesssim \|b\|_{L^1}^2(\|v_{0,j}\|_{\dot{H}^s}+\|g_j\|_{\B^{s-1/2,2}_1(L^2_T)}),
\end{align*}
and
\begin{align*}\|\partial_x(b_{\ll j}v_j)+b_{\ll
j}\partial_xv_j\|_{\B^{s-1/2,2}_1(L^2_T)} &\lesssim 2^{j(s+1/2)}\|b_{\ll j}\|_{L^1}\|v_j\|_{L^\infty_xL^2_T}\\
& \lesssim \|b\|_{L^1}(\|v_{0,j}\|_{\dot{H}^s}+\|g_j\|_{\B^{s-1/2,2}_1(L^2_T)}).
\end{align*}
Therefore,
\begin{equation}\label{est-vj0}\|v_j\|_{\B^{s-1/4,2}_4(L^\infty_T)}\lesssim
\|v_{0,j}\|_{\dot{H}^s}+\|g_j\|_{\B^{s-1/2,2}_1(L^2_T)}
\end{equation}
and the claim follows by interpolation between (\ref{est-vj0}) and (\ref{est-vj1}).
\end{proof}

We are now ready to prove the main linear estimate on equation (\ref{gbo2}).

\begin{proposition}\label{prop-estlin} Let $u$ be a solution of (\ref{gbo2}) with
initial data $u_0\in \dot{H}^s\cap \dot{H}^{s_k}$, $s\in\R$. Then
there exists $T=T(u_0)>0$ and $C=C(u_0)$ such that on $[-T,+T]$,
$$\|u\|_\dX\leq C\|u_0\|_{\dot{H}^s}+C\|f\|_{\dot{\N}^{s,1}}.$$
\end{proposition}

\begin{proof}

Using that
$|P_+u_j|=|P_-u_j|$ (since $u$ is real) and Lemma \ref{lem-vj}, we infer
\begin{align*} \|u_j\|_\dX\lesssim \|v_j\|_\dX &\lesssim
\|Q_ju_0\|_\dX+\|f_j\|_{\dot{\N}^{s,1}}+\|\partial_x (u_{0,\ll j})^k v_j\|_{\dot{\N}^{s,1}}+\|(u_{0,\ll j})^{2k}v_j\|_{\dot{\N}^{s,1}}\\ &\quad
+\left\|\partial_x[((u_{0,\ll j})^k-(u_{L,\ll j})^k)u_j]\right\|_{\dot{\N}^{s,1}}+
\left\|\partial_x[Q_j,(u_{L,\ll
j})^k]u_{\sim j}\right\|_{\dot{\N}^{s,1}}\\ &= \|Q_ju_0\|_\dX+\|f_j\|_{\dot{\N}^{s,1}}+A+B+C+D.
\end{align*}
We bound $A$ by \begin{align*}A &\lesssim
2^{j(s-1/2)}\|\partial_x(u_{0,\ll j})^kv_j\|_{L^1_xL^2_T}\lesssim
2^{-j}\|\partial_x(u_{0,\ll
j})^k\|_{L^1}2^{j(s+1/2)}\|v_j\|_{L^\infty_xL^2_T}\\ &\lesssim
2^{-j}\|\partial_x(u_{0,\ll j})^k\|_{L^1}\|v_j\|_\dX.
\end{align*}
As previously, $2^{-j}\|\partial_x(u_{0,\ll
j})^k\|_{L^1}$ can be made  as small as needed by choosing the implicit
constant $J=J(u_0)$ in $u_{0,\ll j}$ large enough:
$$2^{-j}\|\partial_x(u_{0,<j-J})^k\|_{L^1}\lesssim
2^{-j}2^{j-J}\|u_{0}\|_{L^k}^k\lesssim c(u_0)2^{-J}\ll 1.$$ One proceeds similarly for $B$:
\begin{align*}B &\lesssim 2^{j(s-1/2)}\|(u_{0,\ll j})^{2k}v_j\|_{L^1_xL^2_T}\\ &\lesssim
2^{-j}\|(u_{0,\ll j})^k\|_{L^\infty}\|u_{0}\|_{L^{k}}^{k} 2^{j(s+1/2)}\|v_j\|_{L^\infty_xL^2_T}
\\ & \ll \|v_j\|_\dX.
\end{align*}
 Now we
estimate $C$:
\begin{align*}
C &\lesssim 2^{j(s-1/2)}\|\partial_x[((u_{0,\ll j})^k-(u_{L,\ll j})^k) u_j]\|_{L^1_xL^2_T}\\ &\lesssim
2^{j(s+1/2)}\|(u_{0,\ll j})^k-(u_{L,\ll j})^k\|_{L^1_xL^2_T}\|u_j\|_{L^\infty_xL^2_T}\\ &\lesssim
\|u_0-u_L\|_{L^k_xL^\infty_T}(\|u_0\|_{L^k}^{k-1}+\|u_L\|_{L^k_xL^\infty_T}^{k-1})\|u_j\|_\dX\\ &\ll\|u_j\|_\dX
\end{align*}
by Lemma \ref{lemsmall}. Finally we deal with term $D$. By commutator lemma
(Lemma 2.4 in \cite{MR2227135}), we get
\begin{align*}
D &\lesssim 2^{j(s-1/2)}\|\partial_x[Q_j, (u_{L,\ll j})^k] u_j\|_{L^1_xL^2_T}\\
&\lesssim 2^{j(s-1/2)}2^{-j}\|\partial_x(u_{L,\ll j})^k\|_{L^{\frac{4}{4-\eps}}_xL^{\frac 2\eps}_T}\|\partial_xu_j\|_{L^{\frac 4\eps}_xL^{\frac 2{1-\eps}}_T}\\
&\lesssim 2^{-j}2^{3j\eps/4}\|\partial_x(u_{L,\ll j})^k\|_{L^{\frac{4}{4-\eps}}_xL^{\frac 2\eps}_T} \|u_j\|_{\dot{\S}^{s,1-\eps}}\\
&\lesssim 2^{-j}2^{3j\eps/4} \|\partial_x u_{L,\ll j}\|_{L^{(\frac 1k-\frac \eps 4)^{-1}}_xL^{\frac 2\eps}_T} \|u_{L,\ll j}\|_{L^k_xL^\infty_T}^{k-1}\|u_j\|_\dX\\
&\lesssim \|D_x^{3\eps/4}u_{L,\ll j}\|_{L^{(\frac 1k-\frac \eps 4)^{-1}}_xL^{\frac 2\eps}_T}\|u_j\|_\dX.
\end{align*}
Since the triplet $(\frac{3\eps}{4}-s_k,(\frac 1k-\frac \eps 4)^{-1},\frac 2\eps)$ is
1-admissible, for any $\eta>0$, we can choose  $T>0$ small enough such that
$$\|D_x^{3\eps/4}u_{L}\|_{L^{(\frac 1k-\frac \eps
4)^{-1}}_xL^{\frac 2\eps}_T}<\eta.$$ Gathering all these estimates
we infer $$\|u_j\|_\dX\lesssim
\|Q_ju_0\|_\dX+\|f_j\|_{\dot{\N}^{s,1}}.$$
Summing  this inequality over $j$ finishes the proof of Proposition \ref{prop-estlin}.
\end{proof}

We also need $L^k_xL^\infty_T$-norm estimates.
\begin{proposition}\label{prop-lkli} Let $u$ be a solution of (\ref{gbo2}) with
initial data $u_0\in\dot{H}^{s_k}$. Then there exists $T>0$ and
$C=C(u_0)$ such that
$$\|u\|_{L^k_xL^\infty_T}\leq
C\|u_0\|_{\dot{H}^{s_k}}+C\|f\|_{\dot{\N}^{s_k,1}}.$$ Moreover, if $u_0\in \dot{H}^s\cap\dot{H}^{s_k}$, $s\in \R$, then
\begin{equation}\label{est-lihs}\|u\|_{L^\infty_T\dot{H}^{s}_x}\leq
C\|u_0\|_{\dot{H}^{s}}+C\|f\|_{\dot{\N}^{s,1}}.\end{equation}
\end{proposition}

\begin{proof} We can rewrite our equation as
$$u=u_L-\int_0^tV(t-t')(f-\pi(u_L,u))dt'.$$
By virtue of Lemma \ref{lem-nohombes} and Lemma \ref{lem-estnohom},
we deduce
$$\|u\|_{L^k_xL^\infty_T}\lesssim
\|u_0\|_{\dot{H}^{s_k}}+\|f\|_{\dot{\N}^{s_k,1}}+\|\pi(u_L,u)\|_{\dot{\N}^{s_k,1}},$$
$$\|u\|_{L^\infty_T\dot{H}^{s}_x} \lesssim
\|u_0\|_{\dot{H}^{s}}+\|f\|_{\dot{\N}^{s,1}}+\|\pi(u_L,u)\|_{\dot{\N}^{s,1}}.$$
Then we get
\begin{align*}
\|\pi(u_L,u)\|_{\dot{\N}^{s,1}} &\lesssim \Big(\sum_j\big[2^{j(s+1/2)}\|(u_{L,\ll j})^ku_{\sim j}\|_{L^1_xL^2_T}\big]^2\Big)^{1/2}\\
&\lesssim \|u_L\|_{L^k_xL^\infty_T}^k\Big(\sum_j\big[2^{j(s+1/2)}\|u_{\sim j}\|_{L^\infty_xL^2_T}\big]^2\Big)^{1/2}\\
&\lesssim \|u_0\|_{\dot{H}^{s_k}}^k\|u\|_{\dot{X}^{s}}\\
&\lesssim C(u_0)(\|u_0\|_{\dot{H}^{s}}+\|f\|_{\dot{\N}^{s,1}})
\end{align*}
by Proposition \ref{prop-estlin}.
\end{proof}

\section{Well-posedness for $k\geq 4$}\label{sec-kgeq4}
\subsection{Nonlinear estimates}

Now we estimate the right-hand side of (\ref{gbo2}) in
$\dot{\N}^{s,1}$-norm.

\begin{proposition}\label{prop-estnl}
For any $u\in\dX\cap L^k_xL^\infty_T$, we have
$$\|\pi(u_L,u)-\pi(u,u)\|_{\dot{\N}^{s,1}}\lesssim
\|u_L-u\|_{L^k_xL^\infty_T}\big(\|u_L\|_{L^k_xL^\infty_T}^{k-1}+\|u\|_{L^k_xL^\infty_T}^{k-1}\big)\|u\|_\dX$$
and
$$\|g\|_{\dot{\N}^{s,1}}\lesssim \|u\|_{L^k_xL^\infty_T}^{k-1}\|u\|_\dX^2.$$
\end{proposition}

\begin{proof} Set $u_j=Q_ju$, $u_{\ll j}=Q_{\ll j}u$, etc. Then:
\begin{align*}
\|\pi(u_L,u)-\pi(u,u)\|_{\dot{\N}^{s,1}} &\lesssim \Big(\sum_j\big[2^{j(s-1/2)}\|\partial_x[((u_{L,\ll j})^k-(u_{\ll j})^k)u_{\sim j}]\|_{L^1_xL^2_T}\big]^2\Big)^{1/2}\\
&\lesssim \Big(\sum_j\big[2^{j(s+1/2)}\|(u_{L,\ll j})^k-(u_{\ll j})^k\|_{L^1_xL^\infty_T}\|u_{\sim j}\|_{L^\infty_xL^2_T}\big]^2\Big)^{1/2}\\
& \lesssim
\|u_L-u\|_{L^k_xL^\infty_T}\big(\|u_L\|_{L^k_xL^\infty_T}^{k-1}+\|u\|_{L^k_xL^\infty_T}^{k-1}\big)\|u\|_\dX.
\end{align*}
We bound the second term by
\begin{align*}
\|g\|_{\dot{\N}^{s,1}} &\lesssim \Big(\sum_j\big[2^{j(s+1/2)}\sum_{r\gtrsim j}\|(u_{\sim r})^2(u_{\lesssim r})^{k-1}\|_{L^1_xL^2_T}\big]^2\Big)^{1/2}\\
&\lesssim \Big(\sum_j\big[\sum_{r\gtrsim j}2^{j(s+1/2)}\|u_{\sim r}\|_{L^{\frac 4\eps}_xL^{\frac{2}{1-\eps}}_T}
\|u_{\sim r}\|_{L^{(\frac 1k-\frac \eps{4})^{-1}}_xL^{\frac 2\eps}_T} \|u_{\lesssim r}\|_{L^k_xL^\infty_T}^{k-1}\big]^2\Big)^{1/2}\\
&\lesssim \|u\|_{L^k_xL^\infty_T}^{k-1}\sup_r2^{3\eps r/4}\|u_{\sim r}\|_{L^{(\frac 1k-\frac \eps{4})^{-1}}_xL^{\frac 2\eps}_T}\\
&\quad \times
\Big(\sum_j\big[\sum_{r\gtrsim j}(2^{(j-r)(s+1/2)})(2^{r(s+1/2-3\eps/4)}\|u_{\sim r}\|_{L^{\frac 4\eps}_xL^{\frac 2{1-\eps}}_T})\big]^2\Big)^{1/2}\\
&\lesssim \|u\|_{L^k_xL^\infty_T}^{k-1}\|u\|_{\dot{S}^{s,\eps}}\Big(\sum_{j\leq 0}2^{j(s+1/2)}\Big)\Big(\sum_j\big[2^{j(s+1/2-3\eps/4)}\|u_{\sim j}\|_
{L^{\frac 4\eps}_xL^{\frac 2{1-\eps}}_T}\big]^2\Big)^{1/2}\\
&\lesssim \|u\|_{L^k_xL^\infty_T}^{k-1}\|u\|_\dX^2
\end{align*}
where we used discrete Young inequality.
\end{proof}

\subsection{Existence in $\dot{H}^{s_k}(\R)$}

Consider the map $F$ defined as
$$F(u)=U(t)u_0-\int_0^tU(t-t')f(t')dt'.$$
We shall contract $F$ in the
intersection of two balls:
$$B_M(u_0,T)=\{u\in \dot{X}^{s_k}\cap L^k_xL^\infty_T :
\|u-u_0\|_{L^k_xL^\infty_T}\leq\delta\}$$ and
$$B_S(u_0,T)=\{u\in \dot{X}^{s_k}\cap L^k_xL^\infty_T :
\|u\|_{\dot{X}^{s_k}}\leq\delta\}$$ endowed with the norm
$$\|u\|_{\dot{Y}_T}=\|u\|_{\dot{X}^{s_k}}+\|u\|_{L^k_xL^\infty_T}.$$ Gathering
Propositions \ref{prop-estlin}, \ref{prop-lkli} and
\ref{prop-estnl} (with $s=s_k$) we find that there
exists $C=C(u_0)>1$ such that
\begin{multline*}\|F(u)\|_{\dot{X}^{s_k}}\leq
C\|U(t)u_0\|_{\dot{X}^{s_k}}
+C(1+\|u-u_0\|_{L^k_xL^\infty_T}^{k-1})\|u\|_{\dot{X}^{s_k}}^2\\
+C(\|u_L-u_0\|_{L^k_xL^\infty_T}+\|u-u_0\|_{L^k_xL^\infty_T})(1+\|u-u_0\|_{L^k_xL^\infty_T}^{k-1})\|u\|_{\dot{X}^{s_k}}
\end{multline*}
and
\begin{multline*}\|F(u)-u_0\|_{L^k_xL^\infty_T}\leq
\|U(t)u_0-u_0\|_{L^k_xL^\infty_T}
+C(1+\|u-u_0\|_{L^k_xL^\infty_T}^{k-1})\|u\|_{\dot{X}^{s_k}}^2\\
+C(\|u_L-u_0\|_{L^k_xL^\infty_T}+\|u-u_0\|_{L^k_xL^\infty_T})(1+\|u-u_0\|_{L^k_xL^\infty_T}^{k-1})\|u\|_{\dot{X}^{s_k}}.
\end{multline*}
We can choose $T=T(u_0)$ small enough so that the quantities
$\|U(t)u_0\|_{\dot{X}^{s_k}}$, $\|u_L-u_0\|_{L^k_xL^\infty_T}$ and
$\|U(t)u_0-u_0\|_{L^k_xL^\infty_T}$ are smaller than $\eps=\frac
1{128C^2}$. Thus if $u\in B_M\cap B_S$ , then
$$\|F(u)\|_{\dot{X}^{s_k}}\leq 4C\eps+4C\delta^2$$
and
$$\|F(u)-u_0\|_{L^k_xL^\infty_T}\leq 4C\eps+4C\delta^2.$$
Now we take $\delta=\frac 1{8C}$  so that $F(u)$ belongs to $B_M\cap B_S$.
In the same way, for any $u_1$ and $u_2$ in $B_M\cap B_S$, one has
\begin{align}
\notag \|F(u_1)-F(u_2)\|_{\dot{Y}_T} &\lesssim \|f(u_1)-f(u_2)\|_{\dot{\N}^{s_k,1}}\\
\notag &\lesssim \|u_L-u_1\|_{L^k_xL^\infty_T}(1+\|u_1\|_{L^k_xL^\infty_T}^{k-1})\|u_1-u_2\|_{\dot{X}^{s_k}}\\ \notag &\quad
+\|u_2\|_{\dot{X}^{s_k}}(\|u_1\|_{L^k_xL^\infty_T}^{k-1}+\|u_2\|_{L^k_xL^\infty_T}^{k-1})\|u_1-u_2\|_{L^k_xL^\infty_T}\\ \notag
&\quad +\|u_1\|_{\dot{X}^{s_k}}^2(\|u_1\|_{L^k_xL^\infty_T}^{k-2}+\|u_2\|_{L^k_xL^\infty_T}^{k-2})\|u_1-u_2\|_{L^k_xL^\infty_T}\\ \label{est-uni} &\quad
+\|u_2\|_{L^k_xL^\infty_T}^{k-1}(\|u_1\|_{\dot{X}^{s_k}}+\|u_2\|_{\dot{X}^{s_k}})\|u_1-u_2\|_{\dot{X}^{s_k}}.
\end{align}
Therefore, $$\|F(u_1)-F(u_2)\|_{\dot{Y}_T} \lesssim
(\eps+\delta)\|u_1-u_2\|_{\dot{Y}_T}$$ and for $\eps,\delta$ small
enough, $F : B_M\cap B_S\rightarrow B_M\cap B_S$ is contractive.
There exists a solution $u$ in $B_M\cap B_S$.

The next step is to show that $u\in \C([-T,+T],\dot{H}^{s_k}(\R))$. Using
(\ref{est-lihs}) and Proposition \ref{prop-estnl}, we obtain that
$u\in L^\infty_T\dot{H}^{s_k}_x$. For any $t_1,t_2\in [0,T]$ with
$t_1<t_2$, writing $u(t)$ as
$$u(t)=V(t-t_1)u(t_1)-\int_{t_1}^tV(t-t')\partial_xu^{k+1}(t')dt',$$
we get
\begin{align*}\|u(t_1)-u(t_2)\|_{\dot{H}^{s_k}} &\lesssim
\sup_{t\in[t_1,t_2]}\|u(t)-u(t_1)\|_{\dot{H}^{s_k}}\\
&\lesssim \sup_{t\in[t_1,t_2]}\|u(t_1)-V(t-t_1)u(t_1)\|_{\dot{H}^{s_k}}\\ &\quad+\Big\|\int_{t_1}^tV(t-t')\partial_x u^{k+1}(t')dt'\Big\|_{L^\infty(t_1,t_2;\dot{H}^{s_k})}\\
&\rightarrow 0
\end{align*}
as $t_1\rightarrow t_2$.

Now consider $u_{0,1},u_{0,2}\in \dot{H}^{s_k}$ two initial data, and
$u_1,u_2\in \dot{Z}_T$ satisfying
$$u_1(t)=U_1(t)u_{0,1}-\int_0^tU_1(t-t')f_1(u_1)(t')dt',$$
$$u_2(t)=U_2(t)u_{0,2}-\int_0^tU_2(t-t')f_2(u_2)(t')dt',$$
where $U_j(t)\varphi$ is solution to
$$\partial_tu+\H\partial_x^2u+\pi(V(t)u_{0,j},u)=0,\quad
u(0)=\varphi$$ and $f_j$ is defined by
$$f_j(u)=\pi(V(t)u_{0,j},u)-\pi(u,u)+g(u).$$
We intend to show that there exists a nondecreasing polynomial function $P\geq 1$ such that
\begin{multline}\label{est-diff}\|u_1-u_2\|_{\dot{Z}_T}\lesssim P(\|u_1\|_{\dot{Z}_T}+\|u_2\|_{\dot{Z}_T})\big[\|u_{0,1}-u_{0,2}\|_{\dot{H}^{s_k}}\\
+(\|u_1\|_{\dot{X}^{s_k}}+\|u_2\|_{\dot{X}^{s_k}})\|u_1-u_2\|_{\dot{Z}_T}\big]
\end{multline}
where  the implicit constant in the inequality may depends on $u_{0,1}$, $u_{0,2}$. Obviously, the uniqueness of the solution to (\ref{gBO}) and the fact that
the flow map is locally Lipschitz from $\dot{H}^{s_k}(\R)$ to $\dot{Z}_T$ follow directly from (\ref{est-diff}).

One has
$$\|U_1(t)u_{0,1}-U_2(t)u_{0,2}\|_{\dot{Z}_T}\lesssim
\|U_1(t)(u_{0,1}-u_{0,2})\|_{\dot{Z}_T}+\|(U_1(t)-U_2(t))u_{0,2}\|_{\dot{Z}_T}.$$
The first term in the right-hand side is bounded by
$\|u_{0,1}-u_{0,2}\|_{\dot{H}^{s_k}}$. To treat the second one, we
note that $(U_1(t)-U_2(t))u_{0,2}$ is solution to
$$\partial_tu+\H\partial_x^2u+\pi(V(t)u_{0,1},u)=\pi(V(t)u_{0,1},U_2(t)u_{0,2})-\pi(V(t)u_{0,2},U_2(t)u_{0,2})$$
with zero initial data. Hence by Propositions \ref{prop-estlin} and \ref{prop-lkli}, \begin{align*}\|(U_1(t)-U_2(t))u_{0,2}\|_{\dot{Z}_T}
&\lesssim
\|\pi(V(t)u_{0,1},U_2(t)u_{0,2})-\pi(V(t)u_{0,2},U_2(t)u_{0,2})\|_{\dot{\N}^{s_k,1}}\\
&\lesssim \|u_{0,1}-u_{0,2}\|_{\dot{H}^{s_k}}.
\end{align*}
We also need to bound
\begin{align}\notag &\Big\|\int_0^t(U_1(t-t')f_1(u_1)-U_2(t-t')f_2(u_2))dt'\Big\|_{\dot{Z}_T}\\
\label{est-lip1} & \lesssim
\Big\|\int_0^tU_1(t-t')(f_1(u_1)-f_1(u_2))dt'\Big\|_{\dot{Z}_T}\\
\label{est-lip2}
 &\quad+\Big\|\int_0^tU_1(t-t')(f_1(u_2)-f_2(u_2))dt'\Big\|_{\dot{Z}_T}\\
\label{est-lip3}
&\quad+ \Big\|\int_0^t(U_1(t-t')-U_2(t-t'))f_2(u_2)dt'\Big\|_{\dot{Z}_T}.
\end{align}
(\ref{est-lip1}) is bounded by
$$(\ref{est-lip1}) \lesssim \|f_1(u_1)-f_1(u_2)\|_{\dot\N^{s_k,1}}$$
and we can use (\ref{est-uni}) to get the desired estimate.
Term (\ref{est-lip2}) is bounded by \begin{align*}(\ref{est-lip2})
 &\lesssim
\|\pi(V(t)u_{0,1},u_2)-\pi(V(t)u_{0,2},u_2)\|_{\dot\N^{s_k,1}}\\
&\lesssim
\|u_2\|_{\dot{X}^{s_k}}\|u_{0,1}-u_{0,2}\|_{\dot{H}^{s_k}}.
\end{align*}
Finally, note that $\int_0^t(U_1(t-t')-U_2(t-t'))f_2(u_2)dt'$ is solution to
$$\partial_tu+\H\partial_x^2u+\pi(V(t)u_{0,1},u)=\pi(V(t)u_{0,2},\psi)-\pi(V(t)u_{0,1},\psi),$$
with zero initial data, and where $\psi=\int_0^tU_2(t-t')f_2(u_2)dt'$. It follows that
\begin{align*}(\ref{est-lip3}) &\lesssim
\|\pi(V(t)u_{0,2},\psi)-\pi(V(t)u_{0,1},\psi)\|_{\dot\N^{s_k,1}}\\
&\lesssim \|\psi\|_{\dot{X}^{s_k}} \|u_{0,1}-u_{0,2}\|_{\dot{H}^{s_k}}\\
&\lesssim (\|u_2\|_{\dot{X}^{s_k}}+\|u_2\|_{\dot{X}^{s_k}}^{k+1})\|u_{0,1}-u_{0,2}\|_{\dot{H}^{s_k}}.
\end{align*}

Gathering all these estimates we obtain (\ref{est-diff}).

\subsection{Existence in ${H}^{s}(\R)$, $s\geq s_k$}
Define the spaces $\X=\dot{X}^{0}\cap\dX$ and
$\N^{s,\theta}=\dot{\N}^{0,\theta}\cap\dot{\N}^{s,\theta}$.

We closely follow the proof of Theorem \ref{th-hom}. We show that $F$ is a contraction in the intersection of
$$B_M(u_0,T)=\{u\in {X}^{s}\cap L^k_xL^\infty_T :
\|u-u_0\|_{L^k_xL^\infty_T}\leq\delta\}$$ and
$$B_S(u_0,T)=\{u\in {X}^{s}\cap L^k_xL^\infty_T :
\|u\|_{{X}^{s}}\leq\delta\}$$ endowed with the norm
$$\|u\|_{Y_T}=\|u\|_{{X}^{s}}+\|u\|_{L^k_xL^\infty_T}.$$
Using Propositions \ref{prop-estlin}, \ref{prop-lkli} and
\ref{prop-estnl} (applied with $s\geq s_k$ and $s=0$) and the embedding $\N^{s,1}\hookrightarrow \dot{\N}^{s_k,1}$ for $s\geq s_k$ we find
\begin{multline*}\|F(u)\|_{{X}^{s}}\leq
C\|U(t)u_0\|_{{X}^{s}}
+C(1+\|u-u_0\|_{L^k_xL^\infty_T}^{k-1})\|u\|_{{X}^{s}}^2\\
+C(\|u_L-u_0\|_{L^k_xL^\infty_T}+\|u-u_0\|_{L^k_xL^\infty_T})(1+\|u-u_0\|_{L^k_xL^\infty_T}^{k-1})\|u\|_{{X}^{s}}
\end{multline*}
and
\begin{multline*}\|F(u)-u_0\|_{L^k_xL^\infty_T}\leq
\|U(t)u_0-u_0\|_{L^k_xL^\infty_T}
+C(1+\|u-u_0\|_{L^k_xL^\infty_T}^{k-1})\|u\|_{{X}^{s}}^2\\
+C(\|u_L-u_0\|_{L^k_xL^\infty_T}+\|u-u_0\|_{L^k_xL^\infty_T})(1+\|u-u_0\|_{L^k_xL^\infty_T}^{k-1})\|u\|_{{X}^{s}}.
\end{multline*}
In the same way, one may show that
\begin{align*}
\|F(u_1)-F(u_2)\|_{Y_T} &\lesssim \|f(u_1)-f(u_2)\|_{{\N}^{s,1}}\\
&\lesssim \|u_L-u_1\|_{L^k_xL^\infty_T}(1+\|u_1\|_{L^k_xL^\infty_T}^{k-1})\|u_1-u_2\|_{{X}^{s}}\\ &\quad
+\|u_2\|_{{X}^{s}}(\|u_1\|_{L^k_xL^\infty_T}^{k-1}+\|u_2\|_{L^k_xL^\infty_T}^{k-1})\|u_1-u_2\|_{L^k_xL^\infty_T}\\
&\quad +\|u_1\|_{{X}^{s}}^2(\|u_1\|_{L^k_xL^\infty_T}^{k-2}+\|u_2\|_{L^k_xL^\infty_T}^{k-2})\|u_1-u_2\|_{L^k_xL^\infty_T}\\ &\quad
+\|u_2\|_{L^k_xL^\infty_T}^{k-1}(\|u_1\|_{{X}^{s}}+\|u_2\|_{{X}^{s}})\|u_1-u_2\|_{{X}^{s}}.
\end{align*}
This proves the existence in $H^s(\R)$. The end of the proof is identical to that of Theorem \ref{th-hom}.

\section{Well-posedness for $k=3$}\label{sec-keq3}

Let $k=3$ and $s>1/3$ be fixed.

The scheme of the proof is the same as for the case $k\geq 4$ with minor modifications. First, in view of Lemma \ref{lem-linl3},
it is clear that Lemma \ref{lemsmall} holds for $k=3$ with $u_0\in \dot{H}^{s_k}$ replaced by $u_0\in H^s$. Next we see that the
$\B^{\frac{3\eps}{4},2}_{(\frac{1}{k}-\frac{\eps}{4})^{-1}}(L^{\frac 2\eps}_T)$ -norm which appears in Proposition \ref{prop-estnl} when
estimating the nonlinear term $g$ is not bounded by the $\dot{S}^{\eps,1}$-norm for $k=3$. So we modify slightly the space $X^s$ by setting
$$X^s=\dot{X}^0\cap \dot{X}^s\cap \B^{\eps,2}_3(L^{\frac{2}{\eps}}_T).$$
On one hand, it is clear from Sobolev inequalities that
$$\|u\|_{\B^{\frac{3\eps}{4},2}_{(\frac{1}{3}-\frac{\eps}{4})^{-1}}(L^{\frac 2\eps}_T)} \lesssim \|u\|_{\B^{\eps,2}_3(L^{\frac{2}{\eps}}_T)}
\lesssim \|u\|_{X^s}.$$
On the other hand, the $\B^{\eps,2}_3(L^{\frac{2}{\eps}}_T)$-norm is acceptable since by (\ref{est3}),
\begin{multline*}\|V(t)\varphi\|_{\B^{\eps,2}_3(L^{\frac{2}{\eps}}_T)}\lesssim \Big(\sum_j4^{j\eps}\|Q_jV(t)\varphi\|_{L^3_xL^\infty_T}^2\Big)^{1/2}\\
\lesssim \Big(\sum_j\|Q_j\varphi\|_{H^{1/3+2\eps}}^2\Big)^{1/2}\lesssim \|\varphi\|_{H^s}\end{multline*}
for $\eps\ll 1$. From this, it is straightforward to check that the subcritical non-homogeneous versions of Propositions \ref{prop-estlin}, \ref{prop-lkli} and
\ref{prop-estnl} are valid whenever $k=3$. This essentially proves Theorem \ref{th-keq3}.

\section*{Acknowledgments}
The author wants to thank Fabrice Planchon for his enthusiastic help and his availability.

\bibliographystyle{plain}
\bibliography{ref}
\end{document}